\documentclass[12pt,onecolumn,draftcls]{IEEEtran}

\usepackage{amsmath,amsthm,amssymb,float,graphicx,epstopdf,geometry,bm}
\usepackage{bbm,bbding,amssymb,pifont,mathrsfs,amsfonts,graphicx,mathtools,color}%
\usepackage{amscd,array,enumerate,dsfont,texdraw,tikz,multicol,bbm}
\allowdisplaybreaks[4]

\newtheorem{lemma}{Lemma}
\newtheorem{assumption}{Assumption}

\newtheorem{remark}{Remark}

\newtheorem{theorem}{Theorem}
\newtheorem{example}{Example}

\ifCLASSINFOpdf
\else
\fi

\begin{document}

\title{
\huge{Exponentially Convergent Algorithm Design for Constrained Distributed Optimization via Non-smooth Approach}
}

\author{Weijian~Li, Xianlin~Zeng, Shu~Liang and Yiguang~Hong

\thanks{W.~Li is with the Department of Automation, University of Science and Technology of China, Hefei 230027, Anhui, China, and is also with the Key Laboratory of Systems and Control, Academy of Mathematics and Systems	Science, Chinese Academy of Sciences, Beijing, 100190, China,
e-mail: \texttt{ustcwjli@mail.ustc.edu.cn}.

X.~Zeng is with the Key Laboratory of Intelligent Control and Decision of Complex Systems, School of Automation, Beijing Institute of Technology, Beijing, 100081, China, e-mail: \texttt{xianlin.zeng@bit.edu.cn}.

S. Liang is with the Key Laboratory of Knowledge Automation for
Industrial Processes of Ministry of Education, School of Automation and
Electrical Engineering, University of Science and Technology Beijing, Beijing
100083, China, and is also with the Institute of Artificial Intelligence,
University of Science and Technology Beijing, Beijing 100083, China.
e-mail: \texttt{sliang@ustb.edu.cn}.

Y.~Hong is with the Key Laboratory of Systems and Control, Academy of Mathematics and Systems
Science, Chinese Academy of Sciences, Beijing, 100190, China, e-mail: \texttt{yghong@iss.ac.cn}.
}}

\maketitle

\begin{abstract}
We consider 
minimizing a sum of non-smooth objective functions with set constraints in a distributed manner.
As to this problem, we propose a distributed algorithm with an exponential convergence rate for the first time.
By the exact penalty method, we reformulate the problem equivalently as a standard distributed one without consensus constraints. Then we design a distributed projected subgradient algorithm with the help of differential inclusions. Furthermore, we show that the algorithm converges to the optimal solution exponentially for strongly convex objective functions.
\end{abstract}


\section{INTRODUCTION}

In the past decade, distributed convex optimization has received intensive research interests due to its broad applications in distributed control \cite{antonell2013interconnected}, recourse allocation \cite{amir2014optimal}, machine learning \cite{li2014communication}, etc. The basic idea is that all agents cooperate to compute an optimal solution with their local information and neighbors' states in a multi-agent network.
A variety of distributed algorithms, either discrete-time or continuous-time, have been proposed for different formulations. For instance, the subgradient method \cite{nedic2009distri}, the dual averaging method \cite{duchi2011dual} and the augmented Lagrangian method \cite{chatzipanagiotis2015augmented} were designed for the unconstrained distributed optimization, while the primal-dual dynamics was explored for constrained distributed problems \cite{liang2017distributed, cherukuri2016asymptotic, zeng2016distributed, zhu2018projected}. Among these formulations, one of the most important one lies in the distributed optimization problem with set constraints \cite{zeng2016distributed, lei2016primal}. As to algorithms, the continuous-time design, including differential equations \cite{wang2010control} and differential inclusions \cite{liu2016collective}, are increasingly popular because they may be implemented by continuous-time physical systems, and moreover, the Lyapunov stability theory provides a powerful tool for their convergence analysis.

Convergence rate is an important criterion to evaluate the performance of a distributed algorithm. Particularly, the exponential convergence is desired in many scenarios. In fact, great efforts have been paid for the exponential convergence of distributed algorithms, especially for the unconstrained distributed optimization \cite{shi2015Extra, ali2017convergence, liang2019exponential}. In \cite{shi2015Extra}, an exact first-order algorithm was proposed with fixed stepsizes and linear convergence rates for the strongly convex objective functions. The linear convergence for alternating direction method of multipliers (ADMM) was discussed in \cite{shi2014on, ali2017convergence}. In \cite{liang2019exponential, yi2019exponential}, the strongly convex assumption was relaxed by metric sub-regularity and restricted secant inequality, respectively, to achieve the exponential convergence for the distributed primal-dual dynamics. As to the distributed problems with affine constraints, some pioneering works have also been done in \cite{cortes2019distributed, yi2016initialization, nedic2018improved}. Based on the saddle point dynamics, a distributed algorithm was proposed with exponential convergence in \cite{cortes2019distributed}. In \cite{deng2017distributed, yi2016initialization}, the distributed continuous-time algorithms for the resource allocation were explored with exponential convergence rates in the absence of set constraints. In \cite{nedic2018improved}, an improved distributed algorithm was designed with geometric rates under the time-varying communication topology.

For the distributed formulations with set constraints, convergence rates have been analyzed for some existing methods. In \cite{liu2017convergence}, the authors reconsidered the consensus-based projected subgradient algorithm, which resulted in a convergence rate of $O(1/{\sqrt k})$ with the non-summable stepsize. Both the distributed continuous-time \cite{zeng2016distributed, zeng2018distributedsiam} and discrete-time \cite{lei2016primal} primal-dual methods were developed, and they converged to an optimal solution with a convergences rate of $O(1/t)$. However, it is still challenging to design a distributed algorithm with exponential convergence rates for these problems.

Inspired by the above observations, we focus on the distributed optimization problem of a sum of non-smooth objective functions with local set constraints. By the exact penalty method, we reformulate the problem equivalently to remove consensus constraints. Furthermore, we explore a distributed algorithm, and provide rigorous proofs for its correctness and convergence properties. Main contributions are summarized as follows.
\begin{enumerate}[a)]
\item  We propose a new distributed continuous-time algorithm by combining the differentiated projected operator and the subgradient method. Note that the algorithm can deal with non-smooth objective functions. Additionally, it is with lower computation and communication burden than the primal-dual algorithm in \cite{zeng2016distributed, lei2016primal}.

\item We show that the proposed algorithm converges to an optimal solution exponentially for strongly convex objective functions. Compared with the existing results, this is the first work with an exponential convergence rate for this problem to our best knowledge.
\end{enumerate}

The rest of this paper is organized as follows. Section \uppercase\expandafter{\romannumeral2} presents necessary preliminaries, while Section \uppercase\expandafter{\romannumeral3} formulates and reformulates our problem. In Section \uppercase\expandafter{\romannumeral4}, a distributed  algorithm is proposed with convergence analysis. Finally, numerical simulations are carried out in Section \uppercase\expandafter{\romannumeral5} and concluding remarks are given in Section \uppercase\expandafter{\romannumeral6}.

\textbf{Notations:}
Let $\mathbb R$ be the set of real numbers, $\mathbb R_{\ge 0}$ be the set of non-negative real numbers and $\mathbb R^m$ be the set of $m$ dimensional real column vectors. Denote $\bm 0$ as vectors with all entries being 0, whose dimensions are indicated
by their subscripts. Denote $x^T$ as the transpose of $x$. Denote $\otimes$ as the Kronecker product. Let $|\cdot |$, $\Vert \cdot \Vert$ be $l_1$-norm and $l_2$-norm of a vector, respectively. For $x, y \in \mathbb R^m$, their Euclidean inner product is denoted by $\langle x,y\rangle$, or sometimes simply $x^T y$. For $x_i\in\mathbb R^m, i\in \{1,\dots,n\}$, we denote $\bm x={\rm col} \{x_1,\dots,x_n\}$ as the vector in $\mathbb R^{mn}$ defined by stacking $x_i$ together in columns.
Define $B(x;r)=\{y~|~\Vert y-x\Vert \le r\}$.
Denote ${\rm int}(\Omega)$ as the interior points of set $\Omega$.
Let  $\Omega_1 \times \Omega_2$ be the Cartesian product of two sets $\Omega_1$ and $\Omega_2$.

\section{PRELIMINARIES}

In this section, we introduce some necessary preliminaries about convex analysis, graph theory and differential inclusion.

\subsection{Convex Analysis}

A set $\Omega \subset \mathbb R^m$ is convex if 
$\lambda x+(1-\lambda)y \in \Omega$ for all $x,y\in \Omega$ and $\lambda \in [0,1]$. 
A function $f:\Omega \rightarrow \mathbb R$ is convex if $\Omega \subset \mathbb R^m$ is a convex set, and moreover, 
\begin{equation*}
f(\theta x+(1-\theta) y) \le \theta f(x)+(1-\theta)f(y), ~\forall x, y\in \Omega, ~\forall \theta \in [0,1].
\end{equation*}
If $g_f(x)\in \mathbb R^m$ satisfies  
\begin{equation} \label{cov}
f(y) \ge f(x)+\langle y-x, g_f(x) \rangle,
~\forall y\in \Omega
\end{equation}
then $g_f(x)\in \partial f(x)$, where $\partial f(x)$ is the subdifferential of $f$ at $x$. Furthermore, $f$ is said to be $\mu$-strongly convex if 
\begin{equation*}
\langle g_f(x)-g_f(y), y-x\rangle \ge \mu \Vert y-x\Vert^2, ~\forall x, y\in\Omega.
\end{equation*}

Let $\Omega \subset \mathbb R^m$ be a convex set. For $x \in \Omega$, the tangent cone to $\Omega$ at $x$ is defined by
\begin{equation}
\begin{aligned}
\label{tancon_def}
\mathcal T_\Omega(x) \triangleq 
\{\lim_{k\to +\infty} \frac {x_k-x}{\tau_k} | &x_k \in \Omega,x_k \to x, \tau_k >0, \tau_k \to 0\},
\end{aligned}
\end{equation} 
while the normal cone to $\Omega$ at $x$ is defined by
\begin{equation}
\label{norcon_def}
\mathcal N_{\Omega}(x) \triangleq \{v\in \mathbb R^m |v^T(y-x)\le 0, ~\forall y\in\Omega\}.
\end{equation}
For $x\in \mathbb R^m$, the projection operator $P_{\Omega}(x)$ is defined by 
\begin{equation} \label{pro}
P_{\Omega}(x)={\rm argmin}_{y\in\Omega} \Vert x-y \Vert. 
\end{equation}
Then
\begin{subequations}
\begin{align}
\langle x-P_{\Omega}(x), z-P_{\Omega}(x) \rangle
\le& ~0,~\forall z\in \Omega \label{pro_ine}\\
x-P_{\Omega}(x) \in& ~\mathcal N_{\Omega}(x). \label{pro_ine2}
\end{align}
\end{subequations}
Referring to \cite{brogliato2006equivalence}, we define the differentiated projection operator $P_{\mathcal T_{\Omega}(x)}(y)$, which can be computed by
\begin{equation} \label{tancom}
P_{\mathcal T_{\Omega}(x)}(y)=y-\beta z^*,
\end{equation}
where $\beta={\rm max}\{0, y^T z^*\}$, and $z^* ={\rm argmax}_{z \in \mathcal N_{\Omega}(x)}\langle y,z\rangle$ such that $\Vert z\Vert=1$.

\subsection{Graph Theory}

Consider a multi-agent network described by an  undirected graph $\mathcal G(\mathcal V, \mathcal E)$, where $\mathcal V$ is the node set and $\mathcal E \subset \mathcal V \times \mathcal V$ is the edge set. Node $j$ is a neighbor of node $i$ if and only if $(i,j)\in \mathcal E$. Denote $\mathcal N_i=\{j|(i,j)\in \mathcal E\}$ as the set of agent $i$'s neighbors.
All nodes can exchange information with their neighbors. A path between nodes $i$ and $j$ is denoted as a sequence of edges $(i, i_1), (i_1, i_2), \dots (i_k, j)$ in
the graph with distinct nodes $i_l \in \mathcal V$. Graph $\mathcal G$ is said to be connected if there exists a path between any pair of distinct nodes.

\subsection{Differential Inclusion}

A differential inclusion is given by
\begin{equation}
\label{dif_in}
\dot x(t) \in \mathcal F(x(t)),
~x(0)=x_0,~t\ge 0
\end{equation}
where $\mathcal F: \mathbb R^m \rightrightarrows \mathbb R^m$ is a set-valued map. A Caratheodory solution to (\ref{dif_in}) defined on $[0,\tau) \subset [0,+\infty)$ is an absolutely continuous function $x:[0,\tau) \rightarrow \mathbb R^m$ satisfying (\ref{dif_in}) for almost all $t\in [0,\tau)$ (in the sense of Lebesgue measure) \cite{cortes2008discontinuous}. The solution $t \rightarrow x(t)$ to (\ref{dif_in}) is a right maximal solution if it cannot be extended in time. Suppose that all the right maximal solutions to (\ref{dif_in}) exists on $[0,+\infty)$. 
If $\bm 0_m \in \mathcal F(x_e)$, then $x_e$ is an equilibrium point of (\ref{dif_in}).
The graph of $\mathcal F$ is defined by ${\rm gph}\mathcal F=\{(x,y)~|~y\in \mathcal F(x),~x \in \mathbb R^m\}$.
The set valued map $\mathcal F$ is said to be upper semi-continuous at $x$ if  there exists $\delta>0$ for all $\epsilon>0$ such that 
\begin{equation*}
\mathcal F(y) \subset \mathcal F(x)+B(0;\epsilon), ~\forall y\in B(x;\delta)
\end{equation*}
and it is upper semi-continuous if it is so at every $x \in \mathbb R^m$.
We collect the following results from \cite[p. 266, p. 267]{aubin1984differential}.
\begin{lemma} \label{exi_lem}
Let $\Omega$ be a closed convex subset of $\mathbb R^m$, and $\mathcal F$ be an upper semi-continuous map with non-empty compact value from $\Omega$ to $\mathbb R^m$. Consider two differential inclusions given by
\begin{subequations}
\begin{align}
\dot x(t) \in&~ \mathcal F(x(t)) -\mathcal N_{\Omega}(x(t)),
~x(0)=x_0 
\label{dyn_1}\\
\dot x(t) \in&~ P_{\mathcal T_\Omega} [\mathcal F(x(t))],  
~~~~~~~~~x(0)=x_0. 
\label{dyn_2}
\end{align}
\end{subequations}
Then the following two statements hold.\\
(\romannumeral1) There is a solution to dynamics (\ref{dyn_1}) if $\mathcal F$ is bounded on $\Omega$. \\
(\romannumeral2) The trajectory $x(t)$ is a solution of (\ref{dyn_1}) if and only if it is a solution of (\ref{dyn_2}). 
\end{lemma}

Consider dynamics (\ref{dif_in}). Let $V$ be a locally Lipschitz continuous function, and $\partial V(x)$ be the 
Clarke generalized gradient of $V$ at $x$. The set-valued Lie derivative for $V$ is defined by
$\mathcal L_{\mathcal F} V(x) \triangleq\{a\in \mathbb R: a = p^T v, p\in \partial V(x),v\in \mathcal F(x)\}$.

The following Barbalat's lemma \cite[Lemma 4.1]{haddadnonlinear} will be used
in the convergence analysis of this paper.
\begin{lemma} \label{Barlem}
Let $\sigma: [0, \infty) \rightarrow \mathbb R$ be a uniformly
continuous function. Suppose that $\lim_{t\to\infty}\int_0^t \sigma(s)ds$ exists, and is finite. Then $\lim_{t\to\infty} \sigma(t)=0$.
\end{lemma}

\section{FORMULATION AND REFORMULATION}

In this section, we formulate the problem, and  reformulate it equivalently by the exact penalty method. In addition, we address the optimal condition for the reformulation.

\subsection{Problem Formulation}

Consider a network of $n$ agents interacting over a undirected graph $\mathcal G(\mathcal V, \mathcal E)$, where $\mathcal V=\{1, \dots, n\}$ and $\mathcal E \subset \mathcal V \times \mathcal V$. For all $i\in \mathcal V$, there are a local (non-smooth) objective function $f_i: \mathbb R^m \rightarrow \mathbb R$, and a local feasible constraint set $\Omega_i \subset \mathbb R^m$. All agents cooperate to reach a consensus solution that minimizes the global objection function $\sum_{i=1}^n f_i(x)$ in the feasible set $\cap_{i=1}^n \Omega_i$. To be strict, the optimization problem can be formulated as
\begin{equation} \label{pri_pro}
\begin{aligned}
{\rm min}~~ \sum\limits_{i=1}^n f_i(x) \quad
{\rm s.t.}~~x\in \cap_{i=1}^n \Omega_i
\end{aligned}
\end{equation}
where  $x \in \mathbb R^m$ is the decision variable to be solved. 


In fact, (\ref{pri_pro}) is a well-known constrained distributed optimization problem.
The non-smooth objective functions appear in a variety of fields including resource allocation, signal processing and machine learning, while the local set constraints are often necessary due to the performance limitations of agents in computation and communication capacities.
Both discrete-time \cite{lei2016primal, liu2017constrained} and continuous-time \cite{zeng2016distributed} algorithms have been explored for (\ref{pri_pro}). However, to our best knowledge, convergence rates of the existing results are no more than $\mathcal O(1/t)$.
The goal of this paper is to design a new distributed algorithm with an exponential convergence rate. 

To ensure the well-posedness of problem (\ref{pri_pro}), the following assumptions are made \cite[Assumption 3.1]{liang2017distributed}, \cite[Assumption 3.1]{zeng2016distributed}.

\begin{assumption} \label{con_ass}
For $i\in\mathcal V$, $f_i$ is convex on an open
set containing $\Omega_i$, and  $\Omega_i$ is convex and compact with $\cap_{i=1}^n {\rm int}(\Omega_i)\neq \emptyset$.
\end{assumption}

\begin{assumption} \label{lip_ass}
For $i\in\mathcal V$, $f_i$ is Lipschitz continuous on $\Omega_i$. There exists a constant $c>0$ such that
\begin{equation}\label{lip_ie}
|f_i(x)-f_i(y)|\le c\Vert x-y\Vert,~\forall x, y\in\Omega_i.
\end{equation}
\end{assumption}

\begin{assumption} \label{gra_ass}
The undirected graph of the multi-agent network is connected.
\end{assumption}

\begin{assumption} \label{sto_con_ass}
For $i\in \mathcal V$, $f_i$ is $\beta$-strongly convex on $\Omega_i$, that is, 
\begin{equation} 
\label{stro_equ}
\langle x-y, g_{f_i}(x)- g_{f_i}(y)\rangle 
\ge \beta \Vert x-y \Vert^2,~\forall x,y\in\Omega_i.
\end{equation}
where $g_{f_i}(x) \in \partial f_i(x)$ and $g_{f_i}(y) \in \partial f_i(y)$.
\end{assumption}

Assumption \ref{con_ass} on feasibility is reasonable to ensure the solvability of (\ref{pri_pro}). Compared with the formulation in \cite{zeng2016distributed}, we suppose that $f_i$ is Lipschitz continuous in Assumption \ref{lip_ass}, which is necessary for the problem reformulation in this work. However, we should note that the assumption is easy to hold in practice, especially when $\Omega_i$ is a compact set. 
Assumption \ref{gra_ass} on the communication topology is broadly employed for all agents obtaining a consensus solution.
Furthermore, Assumption \ref{sto_con_ass} is well-known to guarantee the exponential convergence. Subgradients are utilized in (\ref{stro_equ}) because the non-smooth objective functions are considered.

\subsection{Reformulation}

Notice that (\ref{pri_pro}) is not of a standard distributed structure. Under Assumption \ref{gra_ass}, it is equivalent to
\begin{equation}\label{dis_str}
\begin{aligned}
{\rm min} \quad &\sum\limits_{i=1}^n f_i(x_i)\\
{\rm s.t.}\quad &x_i=x_j,~x_i\in \Omega_i,~i\in \mathcal V, ~j\in \mathcal N_i
\end{aligned}
\end{equation}
where $\mathcal N_i$ is the neighbor set of agent $i$. By the exact penalty method, (\ref{dis_str}) can be cast into
\begin{equation}\label{dis_pen}
\begin{aligned}
{\rm min}\quad& \sum\limits_{i=1}^n f_i(x_i)+ \frac K2\sum\limits_{i=1}^n\sum\limits_{j\in \mathcal N_i} |x_i-x_j|\\
{\rm s.t.}\quad &x_i\in \Omega_i,~i\in \mathcal V, ~j\in \mathcal N_i
\end{aligned}
\end{equation}
where $K \in \mathbb R_{\ge 0}$ is the penalty factor. 

Throughout this paper, we define $\bm x={\rm col}\{x_1,\dots, x_n\}$, $\bar \Omega =\Omega_1 \times \dots \times \Omega_n$, $f(\bm x)= \sum_{i=1}^n f_i(x_i)$, and 
\begin{equation} \label{lag_fun}
\mathcal L(\bm x)=\sum\limits_{i=1}^n f_i(x_i)+ \frac K2\sum\limits_{i=1}^n\sum\limits_{j\in \mathcal N_i} |x_i-x_j|.
\end{equation}
The following lemma addresses the relationship between (\ref{dis_str}) and (\ref{dis_pen}).

\begin{lemma} \label{pro_equ}
Let Assumptions \ref{con_ass}, \ref{lip_ass} and \ref{gra_ass} hold. If the penalty factor satisfies $K>nc$, $\bm x^*={\rm col} \{x_1^*,\dots, x_n^*\}$ is an optimal solution to (\ref{dis_str}) if and only if $\bm x^*$ is an optimal solution to (\ref{dis_pen}).
\end{lemma}

\begin{proof}
Define $\bar x=\frac 1n \sum_{i=1}^n x_i$, $\bar{\bm x}=1_n \otimes {\bar x}$
\begin{equation*} 
\begin{aligned}
h(\bm x)&=\frac 12\sum\limits_{i=1}^n \sum\limits_{j\in \mathcal N_i} |x_i-x_j|  
\ge \frac 12 \sum\limits_{i=1}^n \sum\limits_{j\in \mathcal N_i} \Vert x_i-x_j\Vert,
\end{aligned}
\end{equation*}
and moreover,
\begin{equation*}
\begin{aligned}
d(\bm x)=\sum\limits_{k=1}^n \Vert x_k-\bar x\Vert 
\le \frac 1n\sum\limits_{k=1}^n \sum\limits_{l=1}^n\Vert x_k- x_l\Vert.
\end{aligned}
\end{equation*}

For any $k,l\in \mathcal V$, there must be a path $\mathcal P_{kl} \subset \mathcal E$ connecting $k$ and $l$ due to Assumption \ref{gra_ass}. Then 
\begin{equation}
\label{pen_prf1}
\begin{aligned}
d(\bm x)\le \frac 1{2n} \sum\limits_{k=1}^n \sum\limits_{l=1}^n \sum\limits_{(i,j)\in \mathcal P_{kl}}\Vert x_i- x_j\Vert
\le& \frac 1{n} \sum\limits_{k=1}^n \sum\limits_{l=1}^n h(\bm x)=n h(\bm x).
\end{aligned}
\end{equation}
	
Let $K$ be a scalar such that $K>nc$. As a result, we derive
\begin{equation} \label{equ_ine}
\begin{aligned}
\mathcal L(\bm x) =& f(\bm x)+Kh(\bm x)
\ge f(\bm x)+cd(\bm x) \\
=& f(\bar{\bm x})+f(\bm x)-f(\bar{\bm x})+cd(\bm x) \ge f(\bar{\bm x}).
\end{aligned}
\end{equation}
The first inequality holds due to (\ref{pen_prf1}), while the second inequality holds via (\ref{lip_ie}). It follows from (\ref{equ_ine}) that ${\rm min}~\mathcal L(\bm x) \ge {\rm min}_{x_i=x_j} ~f(\bm x)$. Furthermore, the equality holds if and only if $x_i=\bar x$ for $i \in \mathcal V$. 

Conversely, all minima of ${\rm min}_{x_i=x_j} f(\bm x)$
are also minima of ${\rm min}_{x_i=x_j} \mathcal L(\bm x)$, and they are also of minima of $\mathcal L(\bm x)$ due to (\ref{equ_ine}). Thus, the conclusion follows.
\end{proof}

\begin{remark}
In light of \cite[Theorem 6.9]{ruszczynski2006nonlinear}, $K/2$ can be selected as a constant larger than the infinite norm of the Lagrange multipliers for the equality
constraints in (\ref{dis_str}). However, it is difficult to derive the Lagrange multiplier before solving a dual problem. Under Assumption \ref{lip_ass}, the explicit relationship between the penalty factor, objective functions and the network structure is established in Lemma \ref{pro_equ}.
\end{remark}

In fact, Lemma \ref{pro_equ} provides a sufficient condition for the equivalence between (\ref{dis_str}) and (\ref{dis_pen}). Thus, we can focus on solving (\ref{dis_pen}) without consensus constraints, whose optimal conditions are shown as follows.

\begin{lemma}
\label{opt_lem}
Under Assumptions \ref{con_ass}, \ref{lip_ass} and \ref{gra_ass}, $\bm x^*={\rm col} \{x_1^*,\dots,
x_n^*\}$ is an optimal solution to problem (\ref{dis_pen}) if and only if
\begin{equation} \label{opt_leme}
\bm 0_{m} \in 
P_{\mathcal T_{\Omega_i}(x_i^*)}[-\partial f_i(x_i^*)-K\sum\limits_{j\in \mathcal N_i}
{\rm Sgn}(x_i^*-x_j^*)], ~x_i^* \in \Omega_i
\end{equation}
where $\mathcal T_{\Omega_i}(x_i^*)$, $P_{\mathcal T_{\Omega_i}(x_i^*)}(\cdot)$ are defined by (\ref{tancon_def}) and (\ref{tancom}), respectively,
and ${\rm Sgn}(\cdot)$ is the set-valued sign function with each entry defined by
\begin{equation} \label{sgn_def}
\begin{aligned}
{\rm sgn}(u)\triangleq \partial |u|=
\left\{\begin{split}
&\{1\},   &{\rm if}~ u>0&\\
&[-1,1],  &{\rm if}~ u=0&\\
&\{-1\},  &{\rm if}~ u<0&.
\end{split}\right.
\end{aligned}
\end{equation}
Furthermore, $\bm x^*$ is also an optimal solution to problem (\ref{dis_str}) if  $K>nc$.
\end{lemma}

\begin{proof}
According to the Karush-Kuhn-Tucker (KKT) optimal conditions \cite[Theorem 3.34]{ruszczynski2006nonlinear}, $\bm x^*={\rm col} \{x_1,\dots, x_n^*\}$ is an optimal solution to (\ref{dis_pen}) if and only if
\begin{equation} \label{tan_non}
\bm 0_m \in \partial f_i(x_i^*)+ K\sum\limits_{j\in \mathcal N_i} {\rm Sgn}(x_i^*-x_j^*)+ \mathcal N_{\Omega_i} (x_i^*),
\end{equation}
where $\mathcal N_{\Omega_i} (x_i^*)$ is defined in (\ref{norcon_def}).
It follows from (\ref{norcon_def}) that (\ref{tan_non}) holds if and only if 
(\ref{opt_leme}) holds.
Therefore, $\bm x^*$ is an optimal solution to (\ref{dis_pen}) if and only if (\ref{opt_leme}) holds. By Lemma \ref{pro_equ}, $\bm x^*$ is also an optimal solution to (\ref{dis_str}) if $K>nc$. Thus, the proof is completed.
\end{proof}

\section{MAIN RESULTS}

In this section, we propose a distributed continuous-time projected algorithm for (\ref{dis_pen}) with the help of a differential inclusion and the differentiated projection operator (\ref{tancom}). Then we show convergence properties of the algorithm.

\subsection{Distributed Continuous-Time Projected Algorithm}

For (\ref{dis_pen}), we design a distributed continuous-time projected algorithm as 
\begin{equation}\label{alg_sim}
\dot {\bm x}(t)\in P_{\mathcal T_{\bar \Omega}(\bm x(t))} [-\partial \mathcal L(\bm x(t))],~\bm x(0)=\bm x_0\in \bar \Omega.
\end{equation}

In fact, (\ref{alg_sim}) is inspired by the projected subgradient method \cite{mainge2008strong}. $\mathcal L(\bm x(t))$ is non-smooth due to the non-smoothness of  $f_i$ and $l_1$-norm in (\ref{lag_fun}). Thus, subgradients and differential inclusions are adopted. The projection operator is employed to guarantee the state trajectory of $\bm x(t)$ being in the constraint set $\bar \Omega$. 

For agent $i$, the specific form of (\ref{alg_sim}) is
\begin{equation}\label{alg}
\begin{aligned}
\dot x_i(t) \in P_{\mathcal T_{\Omega_i}(x_i(t))}[-\partial f_i(x_i(t)) -K\sum\limits_{j\in \mathcal N_i}{\rm Sgn}(x_i(t)-x_j(t))], 
~x_i(0)=x_{i,0}\in \Omega_i.
\end{aligned}
\end{equation}

Dynamics (\ref{alg_sim}) is discontinuous because of the projection onto the tangent cone and the non-smoothness of $\mathcal L(\bm x)$. However, its solution is still well-defined in the Caratheodory sense. The reasons are as follows.

Obviously, (\ref{alg_sim}) is of the form (\ref{dyn_2}), where  $\mathcal F(\bm x(t))=-\partial \mathcal L(\bm x(t))$. Notice that $\mathcal L(\bm x(t))$ is convex. Then ${\rm gph}\mathcal F$ is closed.  As a result, $\mathcal F$ is upper semi-continuous with compact convex values. Additionally, $\bar \Omega$ is convex and compact. 
Recalling part (\romannumeral1) of Lemma \ref{exi_lem}, dynamics
\begin{equation}
\dot{\bm x}(t) \in -\partial \mathcal L(\bm x(t))
-\mathcal N_{\bar\Omega}(\bm x)
\end{equation}
has a solution. Therefore, (\ref{alg_sim}) has a solution according to part (\romannumeral2) of Lemma \ref{exi_lem}.


We should note that (\ref{alg}) is a fully distributed algorithm because for each agent, only its local objective function, set constraint and neighbor's states are necessary. By Lemma \ref{opt_lem}, $\bm x^*$ is an optimal solution to (\ref{dis_pen}) if and only if it is an equilibrium point of dynamics (\ref{alg}). For (\ref{alg}), agent $i$ is required to project $-\partial f_i(x_i) -K\sum_{j\in \mathcal N_i}{\rm Sgn}(x_i-x_j)$ onto the tangent cone $\mathcal T_{\Omega_i}(x_i)$. However, the closed form for this projection is not difficult to be computed in practice, especially for some special convex sets such as polyhedrons, Euclidean balls and boxes. Similar projection operators have also been utilized in \cite{ zhu2018projected, zeng2016distributed, yi2016initialization}.

\begin{remark}
One of the most intriguing methods for (\ref{dis_str}) is the distributed projected primal-dual algorithm, which has been discussed in \cite{zeng2016distributed, lei2016primal}.
As a comparison, (\ref{alg}) is with lower communication and computation burden for each agent because dual variables are not necessary.
In fact, similar ideas as (\ref{alg}) have also been explored in \cite{lin2016distributed, zhang2018distributed}. However, the set constraints were not considered in \cite{zhang2018distributed}. We extend the smooth objective functions in \cite{lin2016distributed} into non-smooth cases, and design a different algorithm with an exponential convergence rate.
\end{remark}

\subsection{Convergence Analysis}

It is time to show the convergence for dynamics (\ref{alg}). Before showing the result, we introduce a lemma as follows.

\begin{lemma}\label{solset_lem}
Consider dynamics (\ref{alg}). If $x_i(0)\in \Omega_i$, then $x_i(t) \in \Omega_i$ for all $t\ge 0$.
\end{lemma}
\begin{proof}
For $i\in \mathcal V$, we construct a function as 
$$E_i(x_i(t))=\Vert x_i(t)-P_{\Omega_i}(x_i(t)) \Vert^2.$$ 
Then we have
\begin{equation*}
\nabla E_i(x_i)=  x_i-P_{\Omega_i}(x_i)\in N_{\Omega_i}(x_i).
\end{equation*}
On the other hand,
\begin{equation*}
\begin{aligned}
\dot E_i
=&\langle x_i(t)-P_{\Omega_i}(x_i(t)),~\dot x_i(t) \rangle.
\end{aligned}
\end{equation*}
By (\ref{norcon_def}) and (\ref{alg}), we obtain 
$\dot E_i \le 0$. In other words, $E_i(x_i(t))$ is non-increasing. Furthermore,  $E_i(x_i(0))=0$ due to $x_i(0) \in \Omega_i$. As a result, $E_i(x_i(t))=0$  for all $t \ge 0$,
and then $x_i(t)\in \Omega_i$. Thus, the conclusion follows.
\end{proof}

Lemma \ref{solset_lem} implies that $\bm x(t) \in \bar \Omega$ for all $t>0$ if $\bm x(0)\in \bar \Omega$. The following theorem shows the convergence of  $\bm x(t)$. 

\begin{theorem} 
\label{conve_the}
Consider dynamics (\ref{alg}). Under Assumptions \ref{con_ass}, \ref{lip_ass} and \ref{gra_ass}, $\bm x(t)$ converges to an equilibrium point $\bm x^*$, which is also an optimal solution to  (\ref{dis_pen}). 
\end{theorem}

\begin{proof}
Let $\bm x^*=\{x_1^*, \dots, x_n^*\}$ be an equilibrium point of dynamics (\ref{alg}). Construct a Lyapunov candidate function as
\begin{equation} \label{lya1}
V=\frac 12\sum\limits_{i=1}^n \Vert x_i(t)-x_i^*\Vert^2.
\end{equation}
Clearly, the function $V$ along  (\ref{alg}) satisfies 
\begin{equation*} \label{lya_re}
\begin{aligned}
\mathcal L_{\mathcal F} V&=\{a\in \mathbb R: a=\sum\limits_{i=1}^n\langle x_i-x_i^*, P_{\mathcal T_{\Omega_i}(x_i)}[-\eta_i 
&-K\sum\limits_{j\in \mathcal N_i} \xi_{ij}] \rangle, \eta_i \in \partial f_i(x_i), \xi_{ij} \in {\rm Sgn}(x_i-x_j)\}.
\end{aligned}
\end{equation*}
By (\ref{pro_ine2}) and (\ref{alg}), we obtain
\begin{equation*} 
\begin{aligned}
-\eta_i -K\sum\limits_{j\in \mathcal N_i} \xi_{ij}-\dot x_i \in N_{\Omega_i}(x_i).
\end{aligned}
\end{equation*}
Due to (\ref{norcon_def}), we have
\begin{equation}
\label{dev_in2} 
\begin{aligned}
\langle x_i^*-x_i, -\eta_i -K\sum\limits_{j\in \mathcal N_i} \xi_{ij}-\dot x_i\rangle
\le 0.
\end{aligned}
\end{equation}
As a result,
\begin{equation}\label{dev_in}
a \le \sum\limits_{i=1}^n\langle x_i-x_i^*, -\eta_i -K\sum\limits_{j\in \mathcal N_i} \xi_{ij} \rangle.
\end{equation}
Define
\begin{equation} \label{Wdef}
W(t)=\mathcal L({\bm x}(t)) -\mathcal L(\bm x^*).
\end{equation}
Recalling Lemma \ref{solset_lem} gives $W(\bm x)\ge 0$ because $\bm x(t) \in \bar\Omega$. 
Combining (\ref{cov}) and (\ref{dev_in}), we derive
\begin{equation}  \label{lya_re3}
\begin{aligned}
a \le\sum\limits_{i=1}^n(f_i(x_i^*)-f_i(x_i)) -\frac K2 \sum\limits_{i=1}^n \sum\limits_{j\in \mathcal N_i}|x_i-x_j|
\le -W(t) \le 0.
\end{aligned}
\end{equation}

Since $\mathcal L(\cdot)$ is locally Lipschitz continuous and $\bm x(t)$ is absolutely continuous, $W(t)$ is uniformly continuous in $t$. Then $W(t)$ is Riemman integrable. 
As a result, $W(t)$ is Lebesgue integrable, and the integral is equal to the Riemann integral. It follows from (\ref{lya_re3}) that the Lebesgue integral of $W(t)$ over the infinite interval $[0, +\infty)$ is bounded.

In summary, $\int_0^t W(\bm x(\tau)) d\tau$ exists and is finite.
Furthermore, $\int_0^t W(\tau) d\tau$
is monotonically increasing because $W(t)$ is nonnegative.
Define $\mathcal M=\{\bm x|W(\bm x(t))=0\}$. By (\ref{Wdef}), any $\bm x \in \mathcal M$ is an equilibrium point of dynamics (\ref{alg}). Based on Lemma \ref{Barlem}, we derive $\bm x(t) \to \mathcal M$ as  $t\to \infty$.

Finally, we show that the trajectory $\bm x(t)$ converges to
one of the equilibrium points in $\mathcal M$. There exists a
strictly increasing sequence $\{t_k\}$ with $\lim_{k\to\infty} t_k=+\infty$ such that $\lim_{k\to\infty} \bm x(t_k) =\tilde {\bm x}$ because 
$\lim_{t\to \infty} \bm x(t) \rightarrow \mathcal M$. 
Consider a new Lyapunov function $\tilde V$ defined as (\ref{lya1}) by replacing $\bm x^*$ with $\tilde {\bm x}$.
By a similar procedure as above for $V$, we have $\dot{\tilde V} \le 0$. For any $\epsilon > 0$ , there exists $T$ such that $\tilde V(\bm x(T)) < \epsilon$. Because of $\dot{\tilde V} \le 0$, we obtain
\begin{equation*}
\frac 12\Vert\bm x(t)-\tilde {\bm x} \Vert^2 \le \tilde V(\bm x(T)) < \epsilon,~\forall t\ge T
\end{equation*}
 which implies $\lim_{t\to \infty} \bm x(t) =\tilde {\bm x}$. 
 
According to Lemma \ref{opt_lem}, $\bm x^*$ is an equilibrium point of dynamics (\ref{alg}) if and only if it is an optimal solution to (\ref{dis_pen}). Thus, the proof is completed.
\end{proof}

\begin{remark}
Theorem \ref{conve_the} indicates that dynamics (\ref{alg}) converges to one of the equilibria without Assumption \ref{sto_con_ass}, even in the absence of the strict convexity assumption on objective functions \cite[Assumption 3.3]{zeng2018distributedsiam}, \cite[Remark 4.5]{cherukuri2018role}.
\end{remark}

\subsection{Convergence Rate Analysis}
In this subsection, we analyze the convergence rate for dynamics (\ref{alg}). As is known to all, for a nonlinear dynamics, there is not a unifying framework to estimate its convergence rate. However, in this work, dynamics (\ref{alg}) is carefully designed with the projection onto the tangent cone, which can be easily eliminated via (\ref{dev_in2}). Then Assumption \ref{sto_con_ass} can be naturally employed for the exponential convergence analysis. The main result is shown as follows.

\begin{theorem}
Under Assumptions 1–4, dynamics (\ref{alg}) converges to its equilibrium point exponentially.
\end{theorem}
\begin{proof}
Note that there is only one optimal solution to (\ref{dis_pen}) due to Assumption \ref{sto_con_ass}. Then the equilibrium point of (\ref{alg}) is unique.
According to (\ref{dev_in}), we derive
\begin{equation} 
\begin{aligned}
\label{lya2_1}
a \le \sum\limits_{i=1}^n\langle x_i-x_i^*, -\eta_i -K\sum\limits_{j\in \mathcal N_i} \xi_{ij} \rangle.
\end{aligned}
\end{equation}
For $\eta_i^* \in \partial f_i(x_i^*), \xi_{ij}^* \in {\rm Sgn}(x_i^*-x_j^*)$, $x_i^*$ is an optimal solution to (\ref{dis_pen}) if and only if
\begin{equation}
\label{lya2_var}
\langle x_i-x_i^*,\eta_i^* +K\sum\limits_{j\in \mathcal N_i} \xi_{ij}^* \rangle \ge 0,~\forall x_i \in \Omega_i.
\end{equation}
Substituting (\ref{lya2_var}) into (\ref{lya2_1}), we obtain
\begin{equation*}
\begin{aligned}
a \le& -\sum\limits_{i=1}^n\langle x_i-x_i^*, \eta_i-\eta_i^* +K\sum\limits_{j\in \mathcal N_i} \xi_{ij}-K\sum\limits_{j\in \mathcal N_i} \xi_{ij}^* \rangle \\
\le& -\sum\limits_{i=1}^n \big\langle x_i-x_i^*,\eta_i-\eta_i^*\rangle-\sum\limits_{i=1}^n \sum\limits_{j \in \mathcal N_i}\frac K2 |x_i-x_j|\\
\le& -\sum\limits_{i=1}^n \big\langle x_i-x_i^*,\eta_i-\eta_i^*\rangle
\end{aligned}
\end{equation*}
Recalling (\ref{stro_equ}) gives
\begin{equation*}
a \le -\sum\limits_{i=1}^n\beta\Vert x_i-x_i^* \Vert^2
= -\beta V.
\end{equation*}
As a result, $V(t) \le V(0)e^{-\beta t}$, and $\bm x(t)$ convergence to $\bm x^*$ exponentially. Thus, the conclusion follows.
\end{proof}

\begin{remark}
As to problem (\ref{pri_pro}), a distributed algorithm with an exponential convergence rate is provided for the first time in this work.
On one hand, the consensus constraints in (\ref{dis_str}) are eliminated by the exact penalty method, and then dual variables are not necessary for the algorithm design. On the other hand, properties of the differentiated projection operator (\ref{tancom}) are greatly explored in this work.
\end{remark}

\section{NUMERICAL SIMULATIONS}

In this section, two numerical simulations are carried out for illustration. 
Dynamics (\ref{alg}) is a differential inclusion, and thus, Euler discretization is employed for its numerical implementations in this work. At each step, any subgradient in (\ref{alg}) can be selected. With a fixed stepsize $\alpha$, the $\mathcal O(\alpha)$ approximation is preserved. 

\begin{example}
Here, we provide a numerical example with non-smooth objective functions to verify the convergence of dynamics (\ref{alg}). Consider the multi-agent network with four agents. The communication graph forms a star network. The objective functions and constraints are given by
\begin{equation*}
f_i(x)=|x-a_i|+b_i^Tx,
\end{equation*}
and $\Omega_i=\{x \in \mathbb R^4 ~|~\Vert x-c_i\Vert \le d_i, ~d_i> 0\}$, respectively, where $a_i, b_i, c_i$ and $d_i$ are randomly generated.

Fig. 1 shows the state trajectories of dynamics (\ref{alg}). From the result, (\ref{alg}) converges to an equilibrium point. Moreover, all  agents obtain a consensus solution because $\lim_{t\rightarrow \infty} x_i(t)=x_j(t)$. 

\begin{figure}[H]
\centering
\includegraphics[scale=0.4]{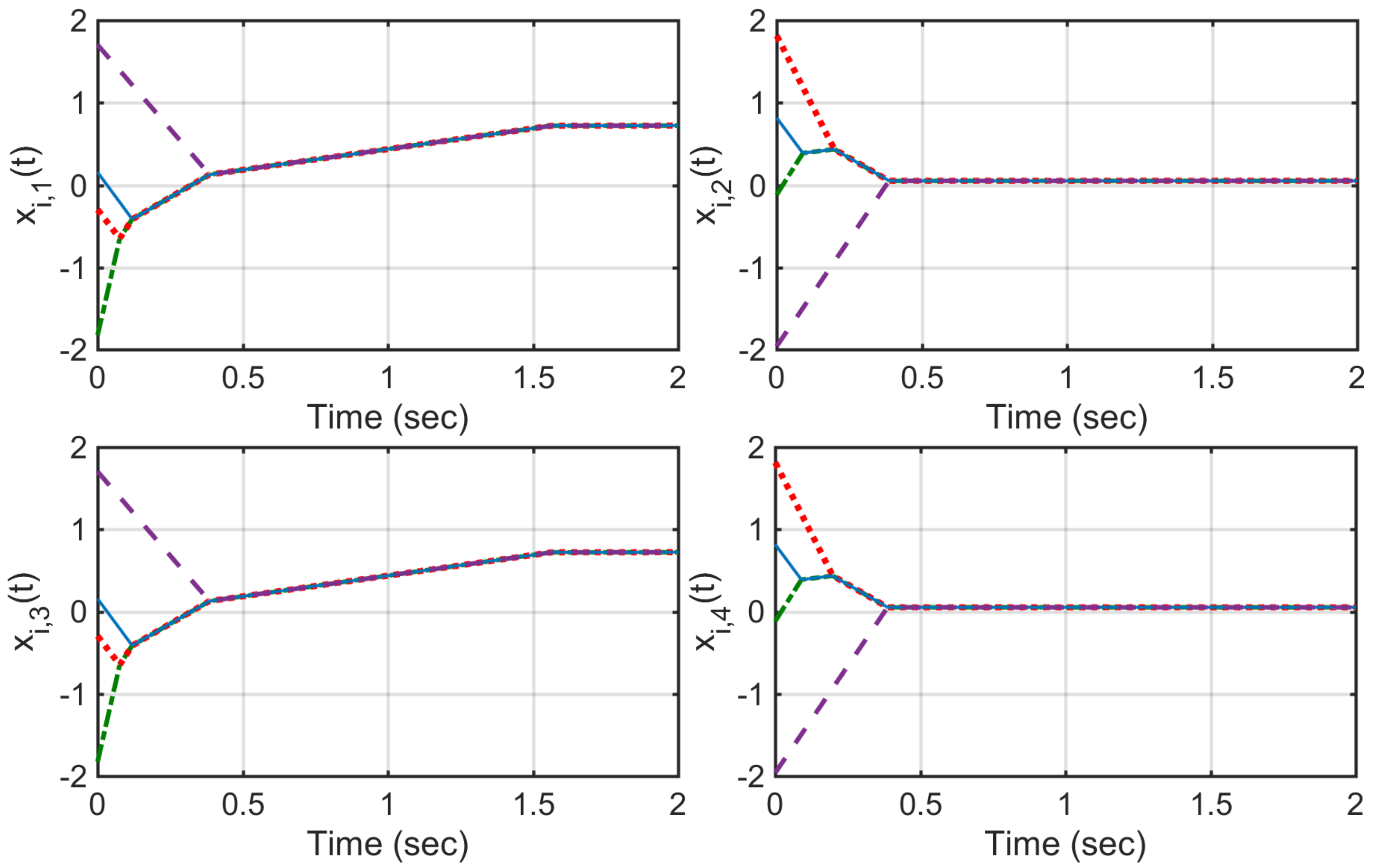}
\caption{State trajectories of algorithm (\ref{alg}).}
\end{figure}

\end{example}

\begin{example}
We carry out a numerical example with strongly convex objective functions to demonstrate exponential convergence of dynamics (\ref{alg}). The multi-agent network consists of thirty agents connected by a cyclic graph. The objective functions are given by
\begin{equation*}
f_i(x)= \frac 12 x^TP_ix+ q_i^T x+r_i|x|,
\end{equation*}
where $P_i \in \mathbb R^{10 \time 10}$ is positive definite,  $q_i\in \mathbb R^{10}$ and $r_i\in \mathbb R_{\ge 0}$. Box constraints are utilized, that is, $\Omega_i=\{x \in \mathbb R^{10} ~|~ l_i \le x\le u_i, u_i> l_i\}$.  Coefficients including $P_i, q_i, r_i, l_i$ and $u_i$ are randomly generated. 

Fig. 2(a) shows the trajectory of $f(\bm x)$, while Fig. 2(b) presents the trajectory of  ${\rm log}(V(t))$. Fig. 2(a) indicates the convergence of (\ref{alg}), and Fig. 2(b) reveals the exponential convergence.

\begin{figure}[H]
\centering
\includegraphics[scale=0.4]{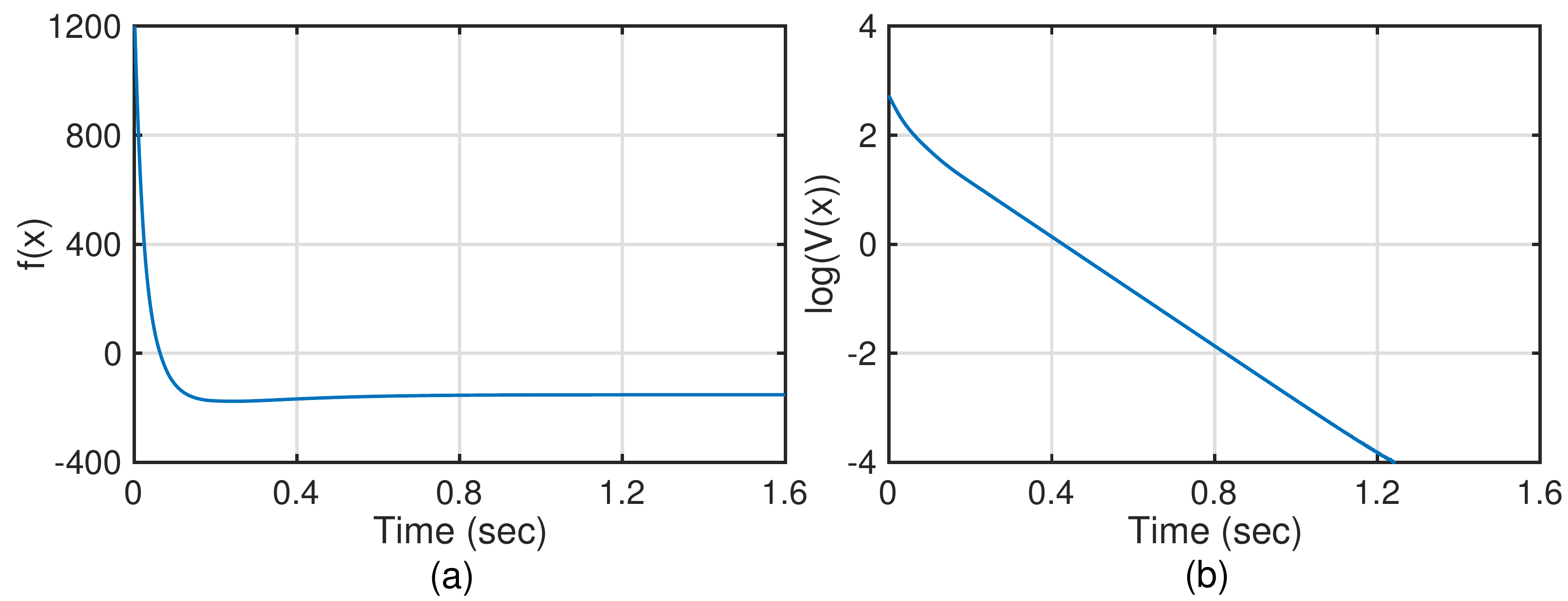}
\caption{Trajectories of $f(\bm x)$ and ${\rm log}(V(t))$ of algorithm (\ref{alg}).}
\end{figure}

\end{example}

\section{CONCLUSION}

This paper aimed at the distributed optimization problem with local set constraints. This problem was equivalently reformulated  without consensus constraints by the penalty method. Resorting to a differentiated projection operation and the subgradient method, a  distributed continuous-time algorithm was proposed. 
The optimal solution could be obtained with an exponential convergence rate for strongly convex objective functions.
Finally, two numerical examples were carried out to verify the results.

\bibliographystyle{plain}
\bibliography{reference.bib}

\end{document}